\documentclass[10pt]{amsart}
\usepackage{graphicx}
\usepackage{latexsym}
\usepackage{amssymb}                
\usepackage{amsmath}
\usepackage{amsthm}
\usepackage{amscd}
\usepackage{enumerate}
\usepackage{tikz}
\usetikzlibrary{arrows}
\usepackage{float}
\usepackage{hyperref}

\theoremstyle{plain}
\newtheorem{thm}{Theorem}[section]
\newtheorem{lem}[thm]{Lemma}
\newtheorem{prop}[thm]{Proposition}
\newtheorem{cor}[thm]{Corollary}

\theoremstyle{definition}
\newtheorem{defi}[thm]{Definition}

\newtheorem*{rem}{Remark}

\numberwithin{thm}{section}

\title{Most Interval Exchanges Have No Roots}
\author{Daniel Bernazzani}

\begin{document}

\begin{abstract}
Let $T$ be an $m$-interval exchange transformation. By the rank of $T$ we mean the dimension of the $\mathbb{Q}$-vector space spanned by the lengths of the exchanged intervals. We prove that if $T$ is minimal and the rank of $T$ is greater than $1+\lfloor m/2 \rfloor$, then $T$ cannot be written as a power of another interval exchange. We also demonstrate that this estimate on the rank cannot be improved.

In the case that $T$ is a minimal 3-interval exchange transformation, we prove a stronger result: $T$ cannot be written as a power of another interval exchange if and only if $T$ satisfies Keane's infinite distinct orbit condition. In the course of proving this result, we give a classification (up to conjugacy) of those minimal interval exchange transformations whose discontinuities all belong to a single orbit.
\end{abstract}

\maketitle
\thispagestyle{empty}

\section{Introduction}

An interval exchange transformation (IET) is a bijective map $T:[0,1)\rightarrow [0,1)$ defined by partitioning the unit interval $[0,1)$ into finitely many subintervals and then rearranging these subintervals by translations.

The formal definition of an IET is given below. The permutation group of the set $\lbrace 1,2,\dots,m \rbrace$ will be denoted by $S_m$. 

\begin{defi}\label{defi:iet}
Fix $m\in \mathbb{N}$. Let $\pi \in S_m$ and let $\lambda=(\lambda_1,\lambda_2,\dots,\lambda_m)$ be a vector in the simplex $$\Delta_{m}=\bigg\lbrace (\lambda_1,\lambda_2,\dots,\lambda_{m})\in \mathbb{R}^m \hspace{1mm} : \hspace{1mm} \lambda_i >0, \hspace{1mm}\sum_i \lambda_i = 1 \bigg\rbrace.$$ Let $$\beta_0=0 \text{ and } \beta_j=\sum_{i=1}^j\lambda_j \text{ for } 1\leq j \leq m.$$ The set $\lbrace \beta_0,\beta_1,\dots,\beta_m \rbrace$ partitions $[0,1)$ into $m$ subintervals of the form $I_j=[\beta_{j-1},\beta_j)$. We can now define a map $T_{(\pi,\lambda)}:[0,1)\rightarrow [0,1)$ by 
$$T_{(\pi,\lambda)}(x)=x - \Bigg(\sum_{i<j}\lambda_i\Bigg) +\Bigg(\sum_{\pi(i)<\pi(j)}\lambda_i\Bigg), \text{ for } x\in I_j.$$ The map $T_{(\pi,\lambda)}$ rearranges the intervals $I_j$ by translations according to the permutation $\pi$. We will refer to a map constructed in this manner as an $m$-IET. For convenience, we sometimes drop the reference to $\pi$ and $\lambda$ and simply denote an IET by a single letter, typically $T$ or $S$. 
\end{defi}  

The dynamical properties of a single IET have been studied extensively. Some early papers in the field are those of Keane \cite{Keane1,Keane2}, Rauzy \cite{Rauzy}, and Veech \cite{Veech_IET}. In 1977, Keane famously conjectured that a typical IET is uniquely ergodic \cite{Keane2}. Keane's conjecture was proven in 1982 by Veech \cite{Veech_Gauss} and Masur \cite{Masur}, who worked independently of one another. Later, Boshernitzan \cite{Bosh_Unique} gave a different proof of Keane's conjecture by showing that most IETs satisfy an explicit Diophantine condition which implies unique ergodicity. More recently, Avila and Forni \cite{Avila_Forni} proved that IETs are typically weakly mixing and Chaika \cite{Chaika_Disjoint} proved that every ergodic transformation is disjoint from almost every IET. For a good introduction to the ergodic theory of IETs, see Viana's survey \cite{Viana}. 

The focus of this paper is different from those mentioned above. The set of all IETs forms a group $\mathbb{G}$ under composition. Several authors, including Novak \cite{Novak1,Novak2,Novak3}, Vorobets \cite{Vorobets}, and Boshernitzan \cite{Bosh_Subgroups}, have investigated the structure of this group. Despite the recent interest, there is still much that is not known about $\mathbb{G}$. For example, it is unknown whether or not a subgroup of $\mathbb{G}$ could be isomorphic to a non-abelian free group. See \cite{Free_Subgroups}, \cite{Juschenko}, and \cite{Novak3} for some results related to this open question. It is also unknown whether or not $\mathbb{G}$ contains any subgroups of intermediate growth. 

\begin{defi}\label{defi:root}
We will denote $n$-fold compositions $T\circ \cdots \circ T$ by $T^n$. We will say that $T$ has an $n^{th}$ \textit{root} in $\mathbb{G}$ if there exists $S\in \mathbb{G}$ such that $T=S^n$. 
\end{defi}

In this paper we will show that a large class of IETs  do not have any nontrival roots in $\mathbb{G}$. Recall that an IET $T$ is said to be minimal if for each $x\in [0,1)$, the orbit $\mathcal{O}_T(x)=\lbrace T^n(x) : n\in \mathbb{Z} \rbrace$ is dense in $[0,1)$. Recall also that an IET $T$ is said to be of rotation type if there exists $\alpha\in \mathbb{R}$ such that  $T(x)=x+\alpha \hspace{2pt}(\text{mod }1)$ for all $x\in [0,1)$.

\begin{thm}\label{thm:roots general}
Let $T$ be a minimal IET which is not of rotation type. Suppose that the lengths of the exchanged subintervals are linearly independent over $\mathbb{Q}$. Then $T$ does not have an $n^{th}$ root in $\mathbb{G}$ for any $n\geq 2$.
\end{thm}

\noindent Theorem \ref{thm:roots general} follows from a more general result which will be stated in the next section of this paper (see Theorem \ref{thm:roots rank}). 

If we restrict our attention to 3-IETs, we can prove a stronger result. We recall the following definition. 
 
\begin{defi}\label{defi:idoc}
Let $T$ be an $m$-IET. Let $\beta_1,\beta_2,\dots,\beta_{m-1}$ be as in Definition \ref{defi:iet}. We say that $T$ satisfies the \textit{infinite distinct orbit condition} (IDOC) if each of the orbits $\mathcal O_T(\beta_1),\mathcal O_T(\beta_2),\dots,\mathcal O_T(\beta_{m-1})$ is infinite and $\mathcal O_T(\beta_i) \cap \mathcal O_T(\beta_j)=\emptyset$ for $i\neq j$.
\end{defi}

\noindent The IDOC was originally formulated by Keane, who showed that any IET which satisfies it and exchanges two or more intervals must be minimal \cite{Keane1}. 

\begin{thm}\label{thm:roots idoc}
Let $T$ be a minimal 3-IET which is not of rotation type. Then $T$ has an $n^{th}$ root in $\mathbb{G}$ for some $n\geq 2$ if and only if $T$ fails to satisfy the infinite distinct orbit condition.
\end{thm}

If $T$ is a 3-IET which is not of rotation type and $\lambda_1,\lambda_2,\lambda_3$ are the lengths of the exchanged subintervals, then it is straightforward to check that the first return map to the interval $[0,1-\lambda_3)$ is given by $x\mapsto x+(\lambda_1 - \lambda_3) \hspace{3pt}(\text{mod }1-\lambda_3)$.  It follows that $T$ is minimal if and only if $\dfrac{\lambda_1 - \lambda_3}{1-\lambda_3} \notin \mathbb{Q}.$ Assuming that $T$ is minimal, $T$ satisfies the IDOC if and only if the orbits of $0$ and $\lambda_1$ under the first return map are distinct. This is amounts to the statement that for every pair of integers $n,m$, we have $n(\lambda_1 -\lambda_3) \neq \lambda_1 + m(1-\lambda_3)$. Thus, it is possible to decide whether or not $T$ has any roots in $\mathbb{G}$, provided one understands the rational dependencies among $\lambda_1,\lambda_2,\lambda_3$.

Our work is related to a result of Novak. Given $T\in \mathbb{G}$, let $C(T)$ denote the centralizer of $T$ in $\mathbb{G}$ and let $\langle T \rangle$ denote the cyclic subgroup generated by $T$. Among other results, Novak shows that if $T$ is minimal and exhibits ``linear discontinuity growth", then the quotient $C(T)/\langle T \rangle$ is finite \cite[Proposition~5.3]{Novak1}. 

Novak's result on centralizers is related to our investigation of the existence of roots in $\mathbb{G}$. Specifically, if $T$ has infinite order in $\mathbb{G}$, but $C(T)/\langle T \rangle$ is finite, then $T$ cannot have $n^{th}$ roots in $\mathbb{G}$ for all sufficiently large $n$. Nevertheless, the assumption that $T$ is minimal and has linear discontinuity growth is not enough to guarantee that $T$ has no roots in $\mathbb{G}$. Indeed, an examination of the proof of Proposition 2.3 of \cite{Novak1} makes it clear that any IET which satisfies the IDOC and which is not of rotation type will exhibit linear discontinuity growth. Let $S$ be a $3$-IET with permutation $(321)$ which satisfies the IDOC. Let $T=S^n$ for some $n\geq 2$. Then $T$ satisfies the IDOC, so $T$ is minimal and has linear discontinuity growth. However, $T$ has an $n^{th}$ root by construction.

Our proof of Theorem \ref{thm:roots idoc} is based on two other results, which are interesting in their own right. In order to describe them, we introduce the concept of a tower over an IET.

Let $T$ be an $m$-IET. Let $I_1,I_2,\dots,I_m$ be the intervals which are exchanged by $T$ and suppose that $f:[0,1)\rightarrow \mathbb{N}$ is constant on each of these intervals, say $f(x)=n_j$ for $x\in I_j$. We can define a new map $T_f$ as follows. The domain will consist of those points of the form $(x,i)$, where $x\in [0,1)$ and $1\leq i\leq f(x)\in \mathbb{N}$. The map $T_f$ is defined by 
$$T_f(x,i)= \begin{cases} (x,i+1) & \text{if }i+1 \leq f(x) \\ (T(x),1) & \text{otherwise} \end{cases}$$ The domain may be visualized as tower over $[0,1)$. The map $T_f$ transports a point up to the next level of the tower, unless the point is already at the top, in which case it is transported back to the first level according to the original map $T$.

\begin{defi}\label{defi:tower}
Let $T$ and $f$ be as in the preceding paragraph. We can view the map $T_f$ as an IET by laying the levels of the tower end to end, and then rescaling so that the total length of the resulting interval is one. We will refer to an IET constructed in this manner as a \textit{tower of type $(n_1,n_2,\dots,n_m)$} over $T$. If $n_1=n_2=\cdots = n_m$, we will refer to $T_f$ as a \textit{tower of constant height $n_1=n_2=\cdots = n_m$}. 
\end{defi}

The above definition is somewhat arbitrary, since there are many ways to arrange the intervals exchanged by a tower inside $[0,1)$, and therefore many ways to view a tower as an IET. Accordingly, we could have used the term ``tower" to refer to a conjugacy class of IETs rather than a single IET. We opted not to do so since it is convenient to have a specific model in mind when referring to a tower. However, the reader should feel free to use either interpretation. To be clear, throughout this paper the word ``conjugate" refers to conjugation in the group $\mathbb{G}$.  

We will refer to a 2-IET with permutation (21) as a rotation. Figure 1 depicts a tower of type $(m,n)$ over a rotation. 

\begin{figure}
\centering
\vspace{10pt}
\begin{tikzpicture}
\draw[|-|, thick] (2,0) -- (6,0) node at (1,0){$m$};

\draw[|-|, thick] (2,-.5) -- (6,-.5) node at (1,-.5){$m-1$};

\draw node at (1,-1.25){$\vdots$};
\draw node at (4,-1.25){$\uparrow$};

\draw[|-|, thick] (2,-2) -- (6,-2) node at (1,-2){$n+1$};

\draw[|-|, thick] (2,-2.5) -- (6,-2.5)node at (1,-2.5){$n$};
\draw[-|, thick] (6,-2.5) -- (8,-2.5);

\draw[|-|, thick] (2,-3) -- (6,-3) node at (1,-3){$n-1$};
\draw[-|, thick] (6,-3) -- (8,-3);

\draw node at (1,-3.75){$\vdots$};
\draw node at (4,-3.75){$\uparrow$};
\draw node at (7,-3.75){$\uparrow$};

\draw[|-|, thick] (2,-4.5) -- (6,-4.5) node at (1,-4.5){$2$};
\draw[-|, thick] (6,-4.5) -- (8,-4.5);

\draw[|-|, thick] (2,-5) -- (6,-5) node at (1,-5){$1$};
\draw[-|, thick] (6,-5) -- (8,-5);

\draw node at (4,-5.5){$I_1$};
\draw node at (7,-5.5){$I_2$};

\end{tikzpicture}
\caption{A tower of type $(m,n)$ over a 2-IET}
\end{figure}
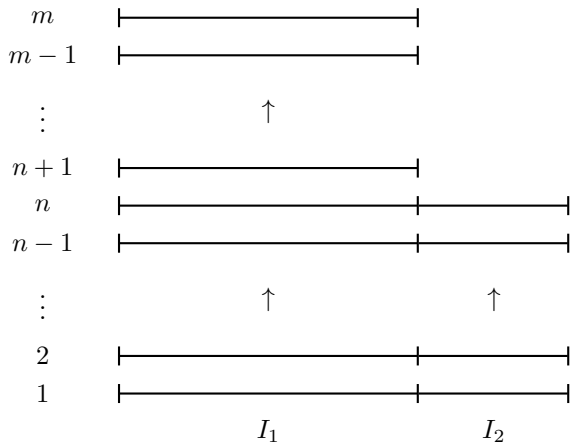   

\begin{thm}\label{thm:one orbit}
Let $T$ be a minimal IET. Suppose that the discontinuities of $T$ all belong to a single orbit. Then $T$ is conjugate to a tower over a minimal rotation.
\end{thm}

Theorem \ref{thm:roots idoc} is closely connected with Theorem \ref{thm:one orbit}. Specifically, if $T$ is a minimal 3-IET which is not of rotation type and which does not satisfy the IDOC, then its two discontinuities $\beta_1$ and $\beta_2$ must belong to the same orbit. Thus, according to Theorem \ref{thm:one orbit}, in order to show that $T$ has a nontrivial root, it suffices to show that towers over minimal rotations always have nontrivial roots. We will do this, and more. Our final result gives a classification of towers over minimal rotations.

\begin{thm}\label{thm:tower classification}
Let $T$ be a tower of type $(m,n)$ over a minimal rotation. Then 
\begin{enumerate}[(1)]
\item
$T$ is conjugate to a minimal rotation if and only if $m$ and $n$ are relatively prime.
\item
If $m$ and $n$ are not relatively prime, and $d>1$ is their greatest common divisor, then $T$ is conjugate to a tower of constant height $d$ over a minimal rotation.
\end{enumerate}    
In either case, $T$ has an $n^{th}$ root in $\mathbb{G}$ for some $n\geq 2$.
\end{thm}

\section{The Rank of an IET}

In this section we define the rank of an IET. This will allow us to state our main result, a generalization of Theorem \ref{thm:roots general}.

\begin{defi}\label{defi:rank}
Let $T$ be an IET. Let $\gamma_1 < \gamma_2 < \cdots < \gamma_{m-1}$ be the points at which $T$ is discontinuous. Let $\gamma_0=0$ and $\gamma_m=1$. Let $l_j = \gamma_j - \gamma_{j-1}$ for $j=1,2,\dots,m$. We will refer to the dimension of the $\mathbb{Q}$-vector space spanned by $l_1,l_2,\dots,l_m$ as the \textit{rank} of $T$. This will be denoted by rank$(T)$.
\end{defi}

The term ``rank" was originally used in this setting by Boshernitzan \cite{Bosh_RankTwo}, who showed that minimal rank two IETs are always uniquely ergodic. Bosherntizan's paper also describes an algorithm which tests a rank two IET for minimality and aperiodicity.

Recall the notation of Definition \ref{defi:iet}. The combinatorial data $(\pi,\lambda)$ which goes into the definition of an IET does not always reflect the number of points at which $T_{(\pi,\lambda)}$ is discontinuous. Since $T$ is a translation on each of the intervals $I_j=[\beta_{j-1},\beta_j)$, it is clear that the discontinuities of $T$ must be among $ \beta_1,\beta_2,\dots,\beta_{m-1}$. However, $T$ may not be discontinuous at all of these points. Whether or not $T$ is continuous at these points depends on the permutation $\pi$. Specifically, if $1\leq i \leq m-1$, then $T$ is discontinuous at $\beta_i$ if and only if $\pi(i+1)\neq \pi(i)+1$. Motivated by this, we make the following definition.

\begin{defi}\label{defi:separating}
We will say that $\pi \in S_m$ is \textit{separating} if $\pi(i+1)\neq \pi(i)+1$ for $1\leq i \leq m-1$.
\end{defi}

For example, the permutation $\tau =(321)\in S_3$ is separating, while the permutation $\sigma = (312) \in S_3$ is not. The following result shows that there is no loss of generality in only considering IETs defined by separating permutations. 

\begin{prop}\label{prop:existence separating}
Let $T$ be an IET with precisely $m-1$ discontinuities. There exists a separating permutation $\pi \in S_m$ and a vector $\lambda \in \Delta_m$, both of which are unique, such that $T=T_{(\pi,\lambda)}$.
\end{prop}
\begin{proof}
Let $\gamma_1 < \gamma_2 < \cdots < \gamma_{m-1}$ be the points at which $T$ is discontinuous. Let $\gamma_0=0$ and $\gamma_m=1$. 

Suppose that $\pi$ and $\lambda$ exist. Since $\pi$ is separating, $T=T_{(\pi,\lambda)}$ must be discontinuous at $\beta_1,\beta_2,\dots,\beta_{m-1}$. Since $T$ has precisely $m-1$ discontinuities, it follows that $\gamma_j=\beta_j$ for $j=1,2,\dots,m-1$. This uniquely specifies $\lambda$, since $\lambda_j=\beta_j - \beta_{j-1} = \gamma_j - \gamma_{j-1}$. It also uniquely specifies $\pi$, since the way in which $T$ rearranges the intervals $[\gamma_{j-1},\gamma_j)$ is intrinsic to the map $T$. This proves that $\pi$ and $\lambda$ are unique if they exist. 

To prove existence, let $\lambda_j= \gamma_j - \gamma_{j-1}$ for $j=1,2,\dots,m$, . Since $\sum_j \lambda_j =1$, the vector $\lambda=(\lambda_1,\lambda_2,\dots,\lambda_m)$ belongs to $\Delta_m$. Let $\pi \in S_m$ be the permutation which describes how the points $\gamma_0,\gamma_1,\dots,\gamma_{m-1}$ are rearranged by $T$. That is, if $T(\gamma_{j_0}) < T(\gamma_{j_1}) < \cdots < T(\gamma_{j_{m-1}})$, then $\pi(j_i + 1)=i+1$ for $i=0,1,\dots,m-1$. By construction, $T=T_{(\pi,\lambda)}$. Since $T$ is discontinuous at $\gamma_1, \gamma_2, \dots , \gamma_{m-1}$, $\pi$ is separating. 
\end{proof}

If $T_{(\pi,\lambda)}$ is an $m$-IET defined by a separating permutation $\pi \in S_m$, then the discontinuities of $T$ are precisely the points $\beta_1,\beta_2,\dots,\beta_{m-1}$. Hence $\text{rank}(T)$ is equal to the dimension of the $\mathbb{Q}$-vector space spanned by $\lambda_1,\lambda_2,\dots,\lambda_m$. On the other hand, if $\pi$ is not separating, then $\text{rank}(T_{\pi,\lambda)})$ may be smaller than the dimension of the $\mathbb{Q}$-vector space spanned by $\lambda_1,\lambda_2,\dots,\lambda_m$. 

We can now state the main result of this paper. Notice that if $T$ is an $m$-IET defined by a separating permutation, then $\text{rank}(T)$ can assume any value from $1$ to $m$.

\begin{thm}\label{thm:roots rank}
Let $T$ be a minimal $m$-IET defined by a separating permutation. If $T$ has a $n^{th}$ root in $\mathbb{G}$ for some $n\geq 2$, then $\text{rank}(T)\leq 1 + \lfloor m/2 \rfloor$. 
\end{thm}

The proof of Theorem \ref{thm:roots rank} will be given in the fifth section of this paper, after we have introduced the necessary machinery. For now, we note that Theorem \ref{thm:roots general} follows directly from Theorem \ref{thm:roots rank}.  
\vspace{5pt}

\noindent \textit{Proof of Theorem \ref{thm:roots general}}. Let $T$ be a minimal IET which is not of rotation type. Suppose that the lengths of the exchanged subintervals are linearly independent over $\mathbb{Q}$. Let $m-1$ be the number of discontinuities of $T$. The assumption that $T$ is not of rotation type implies that $T$ has at least two discontinuities. So $m\geq 3$. Choose $\pi\in S_m$ and $\lambda\in \Delta_m$ according to Proposition \ref{prop:existence separating}. The assumption that the lengths of the exchanged subintervals are linearly independent over $\mathbb{Q}$ implies that the dimension of the $\mathbb{Q}$-vector space spanned by $\lambda_1,\lambda_2,\dots,\lambda_m$ is $m$. Since $\pi$ is separating, $\text{rank}(T)=m$. Since $m\geq 3$, $\text{rank}(T)> 1 + \lfloor m/2 \rfloor$, so Theorem \ref{thm:roots rank} implies that $T$ has no nontrivial roots in $\mathbb{G}$. \qed 
\vspace{5pt}

Recall that a permutation $\pi \in S_m$ is said to be \textit{irreducible} if $\pi(\lbrace 1,2,\dots,k \rbrace) \neq \lbrace 1,2,\dots,k \rbrace$ for any $k<m$. A well-known result of Keane asserts that if $\pi\in S_m$ is irreducible and the coordinates of $\lambda\in \Delta_m$ are linearly independent over $\mathbb{Q}$, then the IET $T_{(\pi,\lambda)}$ is minimal \cite{Keane1}. We will also have $\text{rank}(T_{(\pi,\lambda)})=m$ in this case, provided that $\pi$ is separating. Combining these observations with Theorem \ref{thm:roots rank} proves the following result, which justifies the statement that ``most" IETs do not have any nontrivial roots in $\mathbb{G}$. 

\begin{cor}\label{cor:most}
Let $m\geq 3$. Let $\pi \in S_m$ be separating and irreducible. Let $$A=\lbrace \lambda \in \Delta_m : T_{(\pi,\lambda)} \text{ has no } n^{th} \text{ root in } \mathbb{G} \text{ for any } n\geq 2 \rbrace.$$ Then $A$ is a residual subset of $\Delta_m$ of full Lebesgue measure.
\end{cor}

For any $m\geq 2$ there exist minimal $m$-IETs defined by separating permutations which have nontrivial roots and have rank exactly $1 + \lfloor m/2 \rfloor$. Thus the bound on $\text{rank}(T)$ given in Theorem \ref{thm:roots rank} is optimal. In describing these examples it will be convenient to consider even and odd $m$ separately. It will also be convenient to define IETs on finite intervals other than $[0,1)$. This causes no essential change, since we can always rescale. 

Suppose first that $m$ is even, say $m=2n$. We will define an IET on the interval $[0,n)$. For $j=1,2,\dots,n$, let $I_j=[j-1,j)$. Let $\alpha_1,\alpha_2,\dots,\alpha_{n}\in \mathbb{R}$ be some parameters which will be specified later. Let $R_{1,\alpha_1}$ be the IET which acts on $I_1=[0,1)$ by $x\mapsto x + \alpha_1 \hspace{2pt}(\text{mod }1)$ and which acts as the identity on the intervals $I_2,I_3,\dots,I_n$. Similarly, for $j=2,3,\dots,n$ let $R_{j,\alpha_j}$ be the IET which acts as a rotation by $\alpha_j$ on $I_j$ and leaves the other intervals fixed. Let $P$ be the periodic IET which cyclically permutes the intervals $I_1,I_2,\dots,I_n$.  

If $\alpha_1,\alpha_2,\dots,\alpha_{n}\in \mathbb{R}$ are chosen so that $1,\alpha_1,\alpha_2,\dots,\alpha_{n}$ are linearly independent over $\mathbb{Q}$, then the map $T=PR_{1,\alpha_1}R_{2,\alpha_2} \cdots R_{n,\alpha_n}$ is a minimal $m$-IET. By construction, the permutation which describes $T$ is separating, and $\text{rank}(T)=1 + \lfloor m/2 \rfloor$. We will show that $T$ has an $(n+1)^{st}$ root. In fact, we claim that if $\beta_1,\beta_2,\dots,\beta_n\in \mathbb{R}$ are chosen appropriately, then the map $S=PR_{1,\beta_1}R_{2,\beta_2} \cdots R_{n,\beta_n}$ satisfies $S^{1+n}=T$.

In order to see this, consider how $S$ acts on $I_1$. Applying $S$ once results in this interval being rotated by $\beta_1$ and then moved onto $I_2$. A second application of $S$ rotates the interval by $\beta_2$ and then moves the interval onto $I_3$. After $n$ iterates, the interval has returned to $I_1$ but has been rotated by $\beta_1 + \beta_2 + \cdots + \beta_n$. Applying $S$ one more time causes another rotation by $\beta_1$ and then moves the interval onto $I_2$. So the net effect of $S^{1+n}$ is to rotate $I_1$ by $2\beta_1 + \beta_2 + \cdots + \beta_n$ and then move $I_1$ onto $I_2$. Meanwhile, $T$ rotates $I_1$ by $\alpha_1$ and then moves $I_1$ onto $I_2$. So the action of $S^{1+n}$ on $I_1$ coincides with the action of $T$ if and only if $2\beta_1 + \beta_2 + \cdots + \beta_n\equiv \alpha_1 \hspace{2pt} (\text{mod }1)$.

We can analyze the action of $S^{1+n}$ on the intervals $I_2,I_3,\dots,I_n$ in a similar way. We find that $S^{1+n}=T$ if and only if
\begin{align*}
2\beta_1 + \beta_2 + \beta_3 + \cdots + \beta_n\equiv \alpha_1 \hspace{2pt} (\text{mod }1)\\
\beta_1 + 2\beta_2 + \beta_3 + \cdots + \beta_n\equiv \alpha_2 \hspace{2pt} (\text{mod }1)\\
\beta_1 + \beta_2 + 2\beta_3 + \cdots + \beta_n\equiv \alpha_3 \hspace{2pt} (\text{mod }1)\\
\vdots \hspace{90pt}\\
\beta_1 + \beta_2 + \beta_3 + \cdots + 2\beta_n\equiv \alpha_n \hspace{2pt} (\text{mod }1)
\end{align*}
\noindent Therefore, we can let $(\beta_1,\beta_2,\dots,\beta_n)$ be any solution to the above system of equations.

If $m$ is odd, say $m=2n+1$, we can construct a similar example. This time, we divide the interval $[0,n+1)$ into $n+1$ subintervals $I_1,I_2,\dots,I_{n+1}$ of equal length. We define the rotations $R_{j,\alpha_j}$ as we did previously. Let $P$ be the periodic IET which cyclically permutes the intervals $I_1,I_2,\dots,I_{n+1}$. As before, if $\alpha_1,\alpha_2,\dots,\alpha_{n}\in \mathbb{R}$ are chosen so that $1,\alpha_1,\alpha_2,\dots,\alpha_{n}$ are linearly independent over $\mathbb{Q}$, then $T=PR_{1,\alpha_1}R_{2,\alpha_2} \cdots R_{n,\alpha_n}$ is a minimal $m$-IET. The permutation which describes $T$ is separating, and $\text{rank}(T)=1 + \lfloor m/2 \rfloor$. If $(\beta_1,\beta_2,\dots,\beta_{n+1})$ is a solution to the system of equations
\begin{align*}
2\beta_1 + \beta_2 + \beta_3 + \cdots + \beta_n + \beta_{n+1}\equiv \alpha_1 \hspace{2pt} (\text{mod }1)\\
\beta_1 + 2\beta_2 + \beta_3 + \cdots + \beta_n+ \beta_{n+1}\equiv \alpha_2 \hspace{2pt} (\text{mod }1)\\
\beta_1 + \beta_2 + 2\beta_3 + \cdots + \beta_n+ \beta_{n+1}\equiv \alpha_3 \hspace{2pt} (\text{mod }1)\\
\vdots \hspace{100pt}\\
\beta_1 + \beta_2 + \beta_3 + \cdots + 2\beta_n+ \beta_{n+1}\equiv \alpha_n \hspace{2pt} (\text{mod }1)\\
\beta_1 + \beta_2 + \beta_3 + \cdots + \beta_n+ 2\beta_{n+1}\equiv \hspace{3pt}0 \hspace{3pt}\hspace{3pt} (\text{mod }1)
\end{align*}
\noindent then the map $S=PR_{1,\beta_1}R_{2,\beta_2} \cdots R_{1+n,\beta_{1+n}}$ satisfies $S^{2+n}=T$.
\\

\section{Classification of Towers over Rotations}

In this section, we will prove Theorem \ref{thm:tower classification}. The proof is self-contained. We begin with the following lemma.

\begin{lem}\label{lem:tower induction}
Let $T$ be a tower of type $(m_1,m_2)$ over a rotation. Suppose that $m_1 < m_2$. Then either $T$ is conjugate to a tower of type $(m_1,m_2-m_1)$ over a rotation or $T^{-1}$ is conjugate to a tower of type $(m_2-m_1,m_1)$ over a rotation. If $m_1>m_2$, the analogous result is true. 
\end{lem}
\begin{proof}
Let $I_1$ and $I_2$ be the intervals which are interchanged by the underlying rotation. Let $l_1$ and $l_2$ denote the lengths of these intervals, respectively. The cases $m_2>m_1$ and $m_1>m_2$ are analogous, so we will only consider the case $m_2>m_1$. How we proceed depends on whether or not $m_2-m_1 \geq m_1$.

Suppose first that $m_2-m_1 \geq m_1$. Let $S$ be a tower of type $(m_1,m_2-m_1)$ over the rotation which exchanges two intervals, $J_1$ and $J_2$, of length $l_1+l_2$ and $l_2$, respectively. Divide $J_1$ into two intervals, say $K_1$ and $K_2$, of length $l_2$ and $l_1$, respectively. Let $g:[0,1)\rightarrow [0,1)$ be the IET which is defined as follows. For $0\leq j \leq m_1-1$, $g$ translates the interval $S^j(K_1)$ onto the interval $T^{m_2-m_1+j}(I_2)$. 
For $0\leq j \leq m_1-1$, $g$ translates the interval $S^j(K_2)$ onto the interval $T^{j}(I_1)$. For
for $0\leq j \leq m_2-m_1-1$, $g$ translates the interval $S^j(J_2)$ onto the interval $T^{j}(I_2)$.
Then it is straightforward to verify that $g^{-1}Tg =S$.

Now suppose that $m_2-m_1<m_1$. Let $S$ be a tower of type $(m_2-m_1,m_1)$ over the rotation which exchanges two intervals, $J_1$ and $J_2$, of length $l_2$ and $l_1+l_2$, respectively. Divide $J_2$ into two intervals, say $K_1$ and $K_2$, of length $l_1$ and $l_2$, respectively. Let $g:[0,1)\rightarrow [0,1)$ be the IET which is defined as follows. For $0\leq j \leq m_2-m_1-1$, $g$ translates the interval $S^j(J_1)$ onto the interval $T^{m_2-1-j}(I_2)$. For $0\leq j \leq m_1-1$, $g$ translates the interval $S^j(K_1)$ onto the interval $T^{m_1-1-j}(I_1)$. For $0\leq j \leq m_1-1$, $g$ translates the interval $S^j(K_2)$ onto the interval $T^{m_1-1-j}(I_2)$. One can verify that $g^{-1}T^{-1}g=S$. 
\end{proof} 

\begin{cor}\label{cor:tower conj 1}
Let $T$ be a tower of type $(m_1,m_2)$ over a rotation. Let $d$ be the greatest common divisor of $m_1$ and $m_2$. Then $T$ is conjugate to a tower of constant height $d$ over a rotation.
\end{cor}
\begin{proof}
If $m_1=m_2$ there is nothing to prove. If $m_2 > m_1$, then, according to Lemma \ref{lem:tower induction}, either $T$ is conjugate to a tower of type $(m_1,m_2-m_1)$ or $T^{-1}$ is conjugate to a tower of type $(m_2-m_1,m_1)$. The analogous result holds if $m_1>m_2$. By repeating this several times if necessary, we see that either $T$ or $T^{-1}$ is conjugate to a tower of constant height $d$ over a rotation. The claim follows, since the inverse of a tower of constant height $d$ over a rotation is also a tower of constant height $d$ over a rotation.
\end{proof}

\begin{cor}\label{cor:tower conj 2}
Let $T$ be a tower of type $(m_1,m_2)$ over a rotation. If $m_1$ and $m_2$ are relatively prime, then $T$ is conjugate to a rotation.
\end{cor}
\begin{proof}
This follows from the previous corollary, since a tower of constant height 1 over a rotation is just a rotation.
\end{proof}

\noindent We are now ready to prove Theorem \ref{thm:tower classification}.
\vspace{5pt} 

\noindent \textit{Proof of Theorem \ref{thm:tower classification}}. Let $T$ be a tower of type $(m,n)$ over a minimal rotation. If $m$ and $n$ are relatively prime, Corollary \ref{cor:tower conj 2} tells us that $T$ is conjugate to a rotation. This rotation must be minimal since $T$ is. If $m$ and $n$ are not relatively prime, and $d>1$ is their greatest common divisor, Corollary \ref{cor:tower conj 1} tells us that $T$ is conjugate to a tower of constant height $d$ over a rotation. Once again, this rotation must be minimal, since $T$ is.

We must show that the two cases in the previous paragraph are mutually exclusive. In the first case, $T$ is conjugate to a minimal rotation, so all of its powers are minimal. In the second case, $T^d$ is not minimal, since some conjugate of $T^d$ maps each level of a tower of constant height $d>1$ to itself. This shows that the two cases are not compatible.

It remains to show that $T$ has a nontrivial root in either case. It is clear that a rotation has roots of arbitrary order, so there is nothing to show in the first case. 

Suppose that $T$ is a tower of constant height $d$ over a rotation. By rescaling, we can assume that each level of the tower has length one. Suppose that the underlying rotation is given by $x\mapsto x+\alpha \hspace{2pt}(\text{mod }1)$. 

Let $\beta_1,\beta_2,\dots,\beta_{d}$ be some parameters which will be specified later. For $i=1,2,\dots,d$, let $R_i$ be the IET which acts as a rotation by $\beta_i$ on the $i^{th}$ level of the tower and leaves the other levels fixed. Let $P$ be the periodic IET which cyclically permutes the levels of the tower. We claim that if the parameters $\beta_1,\beta_2,\dots,\beta_{d}$ are chosen appropriately, then $S=PR_1R_2\cdots R_d$ satisfies $S^{1+d}=T$. 

Consider how $S$ acts on the bottom level of the tower. Applying $S$ once results in this interval being rotated by $\beta_1$ and then moved up to the second level. A second application of $S$ rotates the interval by $\beta_2$ and then moves the interval up  to the third level. After $d$ iterates, the interval has returned to the bottom level but has been rotated by $\beta_1 + \beta_2 + \cdots + \beta_d$. Applying $S$ one more time causes another rotation by $\beta_1$ and moves the interval up to the second level. So the net effect of $S^{1+d}$ is to rotate  the bottom level by $2\beta_1 + \beta_2 + \cdots + \beta_d$ and then move the bottom level up to the second level. Meanwhile, $T$ simply moves the bottom level up to the second level. Therefore the action of $S^{1+d}$ on the bottom level of the tower coincides with the action of $T$ if and only if $2\beta_1 + \beta_2 + \cdots + \beta_d\equiv 0 \hspace{2pt} (\text{mod }1)$.

We can analyze the action of $S^{1+d}$ on the other levels of the tower in a similar way. We find that $S^{1+d}=T$ if and only if
\begin{align*}
2\beta_1 + \beta_2 + \beta_3 + \cdots + \beta_d\equiv 0 \hspace{2pt} (\text{mod }1)\\
\beta_1 + 2\beta_2 + \beta_3 + \cdots + \beta_d\equiv 0 \hspace{2pt} (\text{mod }1)\\
\beta_1 + \beta_2 + 2\beta_3 + \cdots + \beta_d\equiv 0 \hspace{2pt} (\text{mod }1)\\
\vdots \hspace{80pt}\\
\beta_1 + \beta_2 + \beta_3 + \cdots + 2\beta_d\equiv \alpha \hspace{2pt} (\text{mod }1)
\end{align*}
So we can let $(\beta_1, \beta_2, \dots, \beta_d)$ be any solution to this system of equations. \qed

\section{The First Return Map}

In this section we review the basic properties of first return maps which will be used in our proofs of Theorems \ref{thm:roots idoc}, \ref{thm:one orbit}, and \ref{thm:roots rank}. At the end of the section, we give the proof of Theorem \ref{thm:one orbit}.

Given an IET $T$, let $D(T)$ denote the set of points at which $T$ is discontinuous. It is well-known that if $T$ is an IET and $[a,b)\subseteq [0,1)$, then the first return map to $[a,b)$ is an IET (up to rescaling). The following lemma is an elaborate formulation of this fact.

\begin{lem}\label{lem:first return}
Let $T$ be an IET. Suppose that $J=[a,b)$ is a subinterval of $[0,1)$ such that $(a,b)\cap D(T) = \emptyset$. Let $P\subseteq J$ consist of those points $x\in (a,b)$ for which there exists an $n\geq 1$ such that $T^j(x)\notin J \cup D(T) \cup \lbrace a,b \rbrace$ for $0<j<n$ and $T^n(x)\in D(T) \cup \lbrace a,b \rbrace.$ The set $P$ is finite. Therefore $P$ partitions $J$ into finitely many subintervals $J_1,J_2,\dots,J_k$. There exist positive integers $m_1,m_2,\dots,m_k$ such that 
\begin{enumerate}[(1)]
\item
for each $i$, the restriction of $T^j$ to $J_i$ is a translation for $1\leq j \leq m_i$;
\item
for each $i$, $T^j(J_i)\cap J = \emptyset$ for $0<j<m_i$;
\item
for each $i$, $T^{m_i}$ translates $J_i$ onto a subinterval of $J$;
\item
the intervals $T^{m_i}(J_i)$, $1\leq i \leq k$, are pairwise disjoint.
\end{enumerate} 
\end{lem}

\noindent For a proof of this result, see e.g. \cite[Lemma~4.2]{Viana}.

\begin{defi}\label{defi:first return}
Let $T$ be an IET. Let $J$ be a subinterval of $[0,1)$ satisfying the hypothesis of Lemma \ref{lem:first return}. For each $x\in J$, let $n_J(x)= \inf \lbrace n \geq 1 :T^n(x)\in J \rbrace$. According to the lemma, $n_J(x)=m_i$ for $x\in J_i$. The map $T_J:J \rightarrow J$ defined by $T_J(x)=T^{n_J(x)}(x)$ is the \textit{first return map} to $J$. The integers $m_1,m_2,\dots,m_k$ are  the \textit{return times}.
\end{defi}

In what follows, we will be concerned with counting the number of distinct orbits to which the discontinuities of an IET belong. Motivated by this, we make the following definition.

\begin{defi}\label{defi:chain}
Let $T$ be an IET. A finite sequence of points $x_1,x_2,$ $\dots,x_k$, $k\geq 1$, in $[0,1)$ will be called a \textit{$T$-chain} if $x_1$ and $x_k$ both belong to $D(T)\cup\lbrace 0 \rbrace$ and $T(x_i)=x_{i+1}$ for $i=1,2,\dots,k-1$. It will sometimes be convenient to refer to the set $\lbrace x_1,x_2,\dots,x_k \rbrace$ itself as a $T$-chain. By a maximal $T$-chain we mean a $T$-chain which is not a proper subset of another $T$-chain.  
\end{defi}

\begin{rem}\label{rem:connection}
According to the above definition, it is possible that a $T$-chain consists of a single point. This is what distinguishes a $T$-chain from a $T$-connection, a term that is used commonly in the literature. 
\end{rem}

Let $T$ be an IET. Clearly every point in $D(T)$ is contained in some $T$-chain. If $T$ is minimal, then every orbit is infinite, but the number of discontinuities is finite, so every $T$-chain is contained in a unique maximal $T$-chain. Any two maximal $T$-chains are clearly disjoint from one another.

\begin{lem}\label{lem:optimal subinterval}
Suppose that $T$ is a minimal IET and that the discontinuities of $T$ belong to precisely $q$ different orbits. Then there is a subinterval $J\subseteq [0,1)$ satisfying the hypothesis of Lemma \ref{lem:first return} for which the first return map $T_J$ is a $(q+1)$-IET.
\end{lem}
\begin{proof}
To say that the discontinuities of $T$ belong to precisely $q$ different orbits is equivalent to saying that there are precisely $q$ distinct maximal $T$-chains. Denote these by $C_1,C_2,\dots,C_q$. Let $C=\bigcup_{i=1}^q C_i$. Let $x\in C$. If $x$ is greater than all other points in $C$, let $y_x=1$. Otherwise, let $y_x$ be the smallest point in $C$ which is greater than $x$. The set $C$ partitions the interval $[0,1)$ into $\vert C \vert$ subintervals, each of which is of the form $I_x=[x,y_x)$, where $x\in C$. Let $J=[a,b)$ be any one of these subintervals. Since $D(T) \subseteq C$, $T$ is continuous at all points in the interior of $J$. Hence $J$ satisifies the hypothesis of Lemma \ref{lem:first return}.

In order to prove that $T_J$ is a $(q+1)$-IET, we need to show that the set $P$ described in Lemma \ref{lem:first return} contains exactly $q$ points.

Let $x\in P$. Then there exists an $n_x\geq 1$ such that $T^j(x)\notin J \cup D(T) \cup \lbrace a,b \rbrace$ for $0<j<n_x$ and $T^{n_x}(x)\in D(T) \cup \lbrace a,b \rbrace \subseteq C$. Since the $T$-chains $C_1,C_2,\dots,C_q$ are all maximal and $x\notin C$, it is clear that $x$ cannot belong to the forward orbit of any point in $C$. Therefore $T^{n_x}(x)$ must be the first point in one of the maximal $T$-chains. Let $y_1,y_2,\dots,y_q$ be the first points in $C_1,C_2,\dots,C_q$, respectively. To prove our claim, we will show that the map $P \rightarrow \lbrace y_1,y_2,\dots,y_q \rbrace$ given by $x\mapsto T^{n_x}(x)$ is a bijection. 

To prove that the map is surjective, it suffices to note that since $T$ is minimal, the backward orbit of each $y_i$ must intersect the interior of $J$. To prove that the map is injective, suppose that $T^{n_x}(x)=T^{n_y}(y)$. Then $x$ and $y$ belong to the same orbit. The defining properties of $n_x$ imply that $n_x \leq n_y$. By symmetry, $n_y \leq n_x$ and consequently $n_x=n_y$. Therefore $x=y$ since $T^{n_x}$ is a bijection.     
\end{proof}
 
We now describe the relationship between first return maps and towers. For convenience, we relax the notion of a tower by allowing the base of the tower to be any finite interval, not necessarily $[0,1)$.
 
\begin{lem}\label{lem:towers and first return}
Let $T$ be a minimal IET. Let $J$ be a subinterval of $[0,1)$ satisfying the hypothesis of Lemma \ref{lem:first return}. Suppose that the first return map $T_J$ is a $k$-IET with return times $m_1,m_2,\dots,m_k$. Then $T$ is conjugate to a tower of type $(m_1,m_2,\dots,m_k)$ over $T_J$.
\end{lem}
\begin{proof}
Let $J_1,J_2,\dots,J_k$ be the subintervals exchanged by $T_J$. Though not explicitly stated in Lemma \ref{lem:first return}, it is not hard to see that the intervals $T^j(J_i)$, for $1\leq i \leq k$ and $0\leq j \leq m_i-1$, are pairwise disjoint. Moreover, the set $A= \bigcup_{i=1}^k \bigcup_{j=0}^{m_i-1} T^j(J_i)$ is clearly $T$-invariant (that is, $T(A)=A$). Therefore $A=[0,1)$ since $T$ is minimal. 

Let $S$ be a tower of type $(m_1,m_2,\dots,m_k)$ over $T_J$. Both $T$ and $S$ act on $[0,1)$ in essentially the same way. For $1\leq i \leq k$, the interval $J_i$ is translated by $T$ onto each of the intervals $T(J_i),T^2(J_i),\dots, T^{m_i-1}(J_i)$ before returning to $J$ according to the map $T_J$. Likewise, for $1\leq i \leq k$, the interval $J_i$ is translated by $S$ onto each of the intervals $S(J_i),S^2(J_i),\dots, S^{m_i-1}(J_i)$ before returning to $J$ according to the map $T_J$. Therefore if $g:[0,1) \rightarrow [0,1)$ is the IET which translates each of the intervals $T^j(J_i)$, for $1\leq i \leq k$ and $0\leq j \leq m_i-1$, onto the corresponding interval $S^j(J_i)$, then it is clear that $T=g^{-1}Sg$.
\end{proof}

\noindent Theorem \ref{thm:one orbit} follows easily from the preceding lemmas.
\vspace{5pt}

\noindent \textit{Proof of Theorem \ref{thm:one orbit}}. Let $T$ be a minimal IET and suppose that the discontinuities of $T$ all belong to a single orbit. By Lemma \ref{lem:optimal subinterval}, there is some interval $J\subseteq [0,1)$ such that $T_J$ is a 2-IET. Since $T$ is minimal, $T_J$ must be a minimal rotation. Lemma \ref{lem:towers and first return} tells us that $T$ is conjugate to a tower over $T_J$. \qed  

\section{Proof of Theorems \ref{thm:roots idoc} and \ref{thm:roots rank}}

In this section we prove Theorems \ref{thm:roots idoc} and \ref{thm:roots rank}. We begin with an important lemma. Recall that $D(T)$ denotes the set of points at which $T$ is discontinuous.

\begin{lem}\label{lem:orbit counting}
Let $T$ be a minimal IET with precisely $m-1$ discontinuities. If there exists an IET $S$ and an integer $n\geq 2$ such that $T=S^n$, then there are at most $\lfloor m/2 \rfloor$ maximal $S$-chains.
\end{lem}
\begin{proof}
Suppose that there is such an $S$. It is clear that $S$ must be minimal. It should be noted that the definition of an IET, together with the fact that $S(0)\neq 0$, implies that $S^{-1}(0)\in D(S)$. Therefore $x_1 \neq 0$ if $\lbrace x_1,x_2,\dots,x_k \rbrace$ is a maximal $S$-chain.

Let $C=\lbrace x_1,x_2,\dots,x_k \rbrace$ be a maximal $S$-chain. Let $p=S^{-(n-1)}(x_1)$. Since $C$ is maximal, none of the points $p,S(p),S^2(p),\dots,S^{n-2}(p)$ can belong to $D(S)\cup\lbrace 0 \rbrace$. Therefore $S^{n-1}$ is continuous at $p$. Since $p\neq 0$, it follows that the restriction of $S^{n-1}$ to some open interval containing $p$ must be a translation. Combining this with the fact that $S$ is discontinuous at $x_1=S^{n-1}(p)$, we see that $S^n=T$ is discontinuous at $p$. Similar reasoning shows that if $x_k \neq 0$, then $S^n=T$ must be discontinuous at $x_k$. 

Suppose that $C_1,C_2,\dots,C_q$ are the distinct maximal $S$-chains. The argument in the preceding paragraph shows that if $0 \not\in C_i=\lbrace x_1,x_2,\dots,x_k \rbrace$ for some $i$, then $C_i$ contributes at least two points to the set $D(T)$, namely $S^{-(n-1)}(x_1)$ and $x_k$. These points must be distinct since $n\geq 2$. If $0 \in C_i=\lbrace x_1,x_2,\dots,x_k \rbrace$, then $C_i$ contributes at least one point to $D(T)$, namely $S^{-(n-1)}(x_1)$. Since there are $q$ maximal $S$-chains, and $0$ belongs to precisely one of them, it follows that $\vert D(T) \vert \geq 2q-1$. On the other hand, $\vert D(T) \vert = m-1$ by assumption. Therefore $q\leq \lfloor m/2 \rfloor$, as claimed.
\end{proof}

\noindent \textit{Proof of Theorem \ref{thm:roots rank}}. Let $T$ be a minimal $m$-IET defined by a separating permutation. Suppose that $S$ is another IET and that there exists a natural number $n\geq 2$ such that $T=S^n$. We have to prove that rank$(T)\leq 1 + \lfloor m/2 \rfloor$. 

It is clear that $D(S^n)\subseteq\bigcup_{i=0}^{n-1}S^{-i}(D(S))$. Using this, it is not hard to verify that rank$(T)\leq$ rank$(S)$. It is also clear that $S$ must be minimal. 

Let $q$ denote the number of distinct $S$-orbits to which the points of $D(S)$ belong. Since $T$ is defined by a separating permutation, $T$ has precisely $m-1$ discontinuities. Lemma \ref{lem:orbit counting} implies that there are at most $\lfloor m/2 \rfloor$ maximal $S$-chains, so $q\leq \lfloor m/2 \rfloor.$

Lemma \ref{lem:optimal subinterval} tells us that there exists some interval $J\subseteq [0,1)$ such that the first return map $S_J$ is a $(q+1)$-IET. Let $J_1,J_2,\dots,J_q,J_{1+q}$ be the subintervals exchanged by $S_J$. Let $m_1,m_2,\dots,m_{1+q}$ be the return times. The intervals $S^j(J_i)$, for $1\leq i \leq 1+q$ and $0\leq j \leq m_i-1$, are pairwise disjoint. Moreover, the set $A= \bigcup_{i=1}^{1+q} \bigcup_{j=0}^{m_i-1} S^j(J_i)$ is $S$-invariant, so $A=[0,1)$ since $S$ is minimal.

We claim that each of the intervals $S^j(J_i)$, for $1\leq i \leq 1+q$ and $0\leq j \leq m_i-1$, must be contained in one of the intervals on which $S$ is continuous. For if not, then some discontinuity of $S$ belongs to the forward $S$-orbit of a point in the interior of one of the intervals $J_i$. This point would then have to belong to the set $P$ described in Lemma \ref{lem:first return}, a contradiction. 

The discussion in the two preceding paragraphs shows that each of the intervals which are exchanged by $S$ is a disjoint union of some of the intervals $S^j(J_i)$, for $1\leq i \leq 1+q$ and $0\leq j \leq m_i-1$. These intervals have at most $1+q$ different lengths, so it follows that $\text{rank}(S)\leq 1+q \leq 1 + \lfloor m/2 \rfloor$. Since $\text{rank}(T)\leq \text{rank}(S)$, this completes the proof. \qed
\vspace{8pt}

\noindent \textit{Proof of Theorem \ref{thm:roots idoc}}. Let $T$ be a minimal 3-IET which is not of rotation type. We have to prove that $T$ has a nontrivial root in $\mathbb{G}$ if and only if $T$ fails to satisfy the IDOC. 

As we mentioned in the introduction, it follows from Theorems \ref{thm:one orbit} and \ref{thm:tower classification} that if $T$ fails to satisfy the IDOC, then $T$ has a root in $\mathbb{G}$. For completeness, we repeat the argument here. Since $T$ does not satisfy the IDOC, both of its discontinuities must belong to the same $T$-orbit. By Theorem \ref{thm:one orbit}, $T$ is conjugate to a tower over a minimal rotation. By Theorem \ref{thm:tower classification}, $T$ has a nontrivial root in $\mathbb{G}$.

Now suppose that $T$ has a nontrivial root in $\mathbb{G}$, say $T=S^n$, where $n\geq 2$. It is clear that $S$ must be minimal. Since $T$ has precisely two discontinuities, Lemma \ref{lem:orbit counting} implies that there is only one maximal $S$-chain. So all of the discontinuities of $S$ belong to a single $S$-orbit. Applying Theorems \ref{thm:one orbit} and \ref{thm:tower classification}, we see that $S$ must be conjugate to either a rotation or a tower of constant height $d>1$ over a rotation. 

In the first case, $T$ is also conjugate to a rotation. Since all rotations commute with one another, the centralizer of $T$ in $\mathbb{G}$ is uncountable. However, Novak has proven that any IET which is minimal and exhibits ``linear discontinuity growth" has a countable centralizer \cite[Proposition~5.3]{Novak1}. The proof of Proposition 2.3 of \cite{Novak1} makes it clear that any IET which satisfies the IDOC and which is not of rotation type will exhibit linear discontinuity growth, and will thus have a countable centralizer. Therefore, since $T$ has an uncountable centralizer, $T$ must not satisfy the IDOC. 

In the second case, observe that $S^d$ is not minimal, since some conjugate of $S^d$ maps each level of a tower of constant height $d>1$ to itself. It follows that $T^d=S^{nd}$ is not minimal. Therefore $T^d$ does not satisfy the IDOC. This implies that $T$ does not satisfy the IDOC either. \qed

\subsection*{Acknowledgments}

I would like to thank Michael Boshernitzan for encouraging me to investigate this topic, for reading many drafts of this paper, and for making a number of helpful suggestions, including a simplification of my original proof of Theorem \ref{thm:roots rank}. I would also like to thank the referee for carefully reading the paper and making several helpful comments.

\end{document}